\documentclass{amsart}
\usepackage{amssymb, amsbsy, amsthm, amsmath, amstext, amsopn}
\usepackage[all]{xy}
\usepackage{amsfonts}
\usepackage{amscd}
\usepackage{mathrsfs}
\usepackage{verbatim}

\newtheorem{thm}{Theorem} [section]
\newtheorem{lemma}[thm]{Lemma}

\newtheorem{corollary}[thm]{Corollary}
\newtheorem{prop}[thm]{Proposition}

\theoremstyle{definition}

\newtheorem{defn}[thm]{Definition}

\theoremstyle{remark}

\newtheorem{remark}[thm]{Remark}

\begin{document}

\numberwithin{equation}{section}

\newcommand{\hs}{\mbox{\hspace{.4em}}}
\newcommand{\ds}{\displaystyle}
\newcommand{\bd}{\begin{displaymath}}
\newcommand{\ed}{\end{displaymath}}
\newcommand{\bcd}{\begin{CD}}
\newcommand{\ecd}{\end{CD}}

\newcommand{\proj}{\operatorname{Proj}}
\newcommand{\bproj}{\underline{\operatorname{Proj}}}
\newcommand{\spec}{\operatorname{Spec}}
\newcommand{\bspec}{\underline{\operatorname{Spec}}}
\newcommand{\pline}{{\mathbf P} ^1}
\newcommand{\pplane}{{\mathbf P}^2}
\newcommand{\coker}{{\operatorname{coker}}}
\newcommand{\ldb}{[[}
\newcommand{\rdb}{]]}

\newcommand{\Sym}{\operatorname{Sym}^{\bullet}}
\newcommand{\Symp}{\operatorname{Sym}}
\newcommand{\Pic}{\operatorname{Pic}}
\newcommand{\AAut}{\operatorname{Aut}}
\newcommand{\PAut}{\operatorname{PAut}}

\newcommand{\too}{\twoheadrightarrow}
\newcommand{\C}{{\mathbf C}}
\newcommand{\cA}{{\mathcal A}}
\newcommand{\cS}{{\mathcal S}}
\newcommand{\cV}{{\mathcal V}}
\newcommand{\cM}{{\mathcal M}}
\newcommand{\bA}{{\mathbf A}}
\newcommand{\cB}{{\mathcal B}}
\newcommand{\cC}{{\mathcal C}}
\newcommand{\cD}{{\mathcal D}}
\newcommand{\D}{{\mathcal D}}
\newcommand{\cs}{{\mathbf C} ^*}
\newcommand{\boldc}{{\mathbf C}}
\newcommand{\cE}{{\mathcal E}}
\newcommand{\cF}{{\mathcal F}}
\newcommand{\cH}{{\mathcal H}}
\newcommand{\cJ}{{\mathcal J}}
\newcommand{\cK}{{\mathcal K}}
\newcommand{\cL}{{\mathcal L}}
\newcommand{\baL}{{\overline{\mathcal L}}}
\newcommand{\M}{{\mathcal M}}
\newcommand{\bM}{{\mathbf M}}
\newcommand{\bm}{{\mathbf m}}
\newcommand{\cN}{{\mathcal N}}
\newcommand{\theo}{\mathcal{O}}
\newcommand{\cP}{{\mathcal P}}
\newcommand{\cR}{{\mathcal R}}
\newcommand{\boldp}{{\mathbf P}}
\newcommand{\boldq}{{\mathbf Q}}
\newcommand{\bbL}{{\mathbf L}}
\newcommand{\cQ}{{\mathcal Q}}
\newcommand{\cO}{{\mathcal O}}
\newcommand{\Oo}{{\mathcal O}}
\newcommand{\OX}{{\Oo_X}}
\newcommand{\OY}{{\Oo_Y}}
\newcommand{\otY}{{\underset{\OY}{\ot}}}
\newcommand{\otX}{{\underset{\OX}{\ot}}}
\newcommand{\cU}{{\mathcal U}}\newcommand{\cX}{{\mathcal X}}
\newcommand{\cW}{{\mathcal W}}
\newcommand{\boldz}{{\mathbf Z}}
\newcommand{\qgr}{\operatorname{q-gr}}
\newcommand{\gr}{\operatorname{gr}}
\newcommand{\coh}{\operatorname{coh}}
\newcommand{\End}{\operatorname{End}}
\newcommand{\Hom}{\operatorname{Hom}}
\newcommand{\uHom}{\underline{\operatorname{Hom}}}
\newcommand{\uHomY}{\uHom_{\OY}}
\newcommand{\uHomX}{\uHom_{\OX}}
\newcommand{\Ext}{\operatorname{Ext}}
\newcommand{\bExt}{\operatorname{\bf{Ext}}}
\newcommand{\Tor}{\operatorname{Tor}}

\newcommand{\inv}{^{-1}}
\newcommand{\airtilde}{\widetilde{\hspace{.5em}}}
\newcommand{\airhat}{\widehat{\hspace{.5em}}}
\newcommand{\nt}{^{\circ}}
\newcommand{\del}{\partial}

\newcommand{\supp}{\operatorname{supp}}
\newcommand{\GK}{\operatorname{GK-dim}}
\newcommand{\hd}{\operatorname{hd}}
\newcommand{\id}{\operatorname{id}}
\newcommand{\res}{\operatorname{res}}
\newcommand{\lrar}{\leadsto}
\newcommand{\im}{\operatorname{Im}}
\newcommand{\HH}{\operatorname{H}}
\newcommand{\TF}{\operatorname{TF}}
\newcommand{\Bun}{\operatorname{Bun}}
\newcommand{\Hilb}{\operatorname{Hilb}}
\newcommand{\Fact}{\operatorname{Fact}}
\newcommand{\F}{\mathcal{F}}
\newcommand{\nthord}{^{(n)}}
\newcommand{\Aut}{\underline{\operatorname{Aut}}}
\newcommand{\Gr}{\operatorname{\bf Gr}}
\newcommand{\GR}{\operatorname{\bf GR}}
\newcommand{\Fr}{\operatorname{Fr}}
\newcommand{\GL}{\operatorname{GL}}
\newcommand{\gl}{\mathfrak{gl}}
\newcommand{\SL}{\operatorname{SL}}
\newcommand{\ff}{\footnote}
\newcommand{\ot}{\otimes}
\def\Ext{\operatorname {Ext}}
\def\Hom{\operatorname {Hom}}
\def\Ind{\operatorname {Ind}}
\def\bbZ{{\mathbb Z}}

\newcommand{\nc}{\newcommand}
\newcommand{\on}{\operatorname}
\nc{\cont}{\on{cont}}
\nc{\rmod}{\on{mod}}
\nc{\Mtil}{\widetilde{M}}
\nc{\wb}{\overline}
\nc{\wt}{\widetilde}
\nc{\wh}{\widehat}
\nc{\sm}{\setminus}
\nc{\mc}{\mathcal}
\nc{\mbb}{\mathbb}
\nc{\Mbar}{\wb{M}}
\nc{\Nbar}{\wb{N}}
\nc{\Mhat}{\wh{M}}
\nc{\pihat}{\wh{\pi}}
\nc{\JYX}{\cJ_{Y\leftarrow X}}
\nc{\phitil}{\wt{\phi}}
\nc{\Qbar}{\wb{Q}}
\nc{\DYX}{\D_{Y\leftarrow X}}
\nc{\DXY}{\D_{X\to Y}}
\nc{\dR}{\stackrel{\bbL}{\underset{\D_X}{\ot}}}
\nc{\Winfi}{\cW_{1+\infty}}
\nc{\K}{{\mc K}}
\nc{\unit}{{\bf \on{unit}}}
\nc{\boxt}{\boxtimes}
\nc{\xarr}{\stackrel{\rightarrow}{x}}

\nc{\aline}{{\mathbb A}^1}
\nc{\Der}{\on{Der}}
\nc{\fM}{{\mathfrak M}}
\nc{\BunDhat}{\wh{\bf{Bun}}_{\D}}
\nc{\BunD}{{\bf Bun}_{\D}}
\nc{\bs}{\backslash}
\nc{\scr}{\mc}

\title{$\cW$-symmetry of the ad\`elic Grassmannian}
\author{David Ben-Zvi}
\address{Department of Mathematics\\University of Texas at Austin\\Austin, TX 78712}
\email{benzvi@math.utexas.edu}
\author{Thomas Nevins}
\address{Department of Mathematics\\University of Illinois at
Urbana-Champaign\\Urbana, IL 61801}
\email{nevins@illinois.edu}

\begin{abstract}
We give a geometric construction of the $\Winfi$ vertex algebra as the infinitesimal
form of a factorization structure on an ad\`elic Grassmannian. 
This gives a concise interpretation of the higher symmetries and B\"acklund-Darboux
transformations
for the KP hierarchy and its multicomponent extensions in terms of a version of
``$\Winfi$-geometry": 
the geometry of $\D$-bundles on smooth curves, or equivalently torsion-free 
sheaves on cuspidal curves.
\end{abstract}

\maketitle

\section{Introduction}
There is a beautiful relationship between the conformal field theory of free
fermions, the KP hierarchy, and the 
geometry of moduli spaces of curves and bundles. In particular, the deformation theory of
curves and of line (or vector) bundles
on them is realized algebraically by the Virasoro and Heisenberg (or Kac-Moody)
algebra symmetries of the free fermion system.

The algebra $\Winfi$ (a central extension of differential operators with Laurent coefficients) 
and its more complicated
nonlinear reductions, the $\cW_n$ algebras,  themselves play a well-established role
as algebras of ``higher symmetries'' in
integrable systems and string theory.  For example, $\cW$-algebras arise as the
additional symmetries of the KP and Toda hierarchies, 
and tau-functions satisfying $\cW$-constraints arise as partition 
functions of matrix models and topological string theories (see e.g. \cite{vM}).  
Moreover, the search for $\cW$-geometry and $\cW$-gravity
generalizing the moduli of curves and topological gravity has also been the topic of
an extensive 
literature (see for example \cite{Hull,Pope,LO}).  

In the present paper, we
establish a ``global'' geometric realization of  these higher $\Winfi$-symmetries. 
More precisely, we give an algebro-geometric description of a factorization space
(Section \ref{factorization}) and prove
(Theorem \ref{comparison}) that it naturally integrates the $\Winfi$-symmetry of the
free fermion system (in a sense described below). 
Our work thus gives a precise analog of earlier work on the Virasoro and Kac-Moody
algebras, in which the role of 
the moduli of curves and bundles is played by the moduli of $\D$-{\em bundles}
(projective
modules over differential operators) on a smooth complex curve $X$, or equivalently by 
the moduli of torsion-free sheaves on {\em cusp quotients} of
$X$.

Beilinson and Drinfeld \cite{chiral} have introduced the notion of
factorization space as a nonlinear (or integrated) geometric counterpart to the
notion of a vertex algebra or 
chiral CFT, which captures the algebraic structure of Hecke correspondences (see \cite{book} for a review). 
The Beilinson-Drinfeld Grassmannian, a factorization space built out of the affine
Grassmannian of
a Lie group $G$, simultaneously encodes the 
infinitesimal (Kac-Moody) and global (Hecke) symmetries of bundles on curves and is
perhaps the central
object in the geometric Langlands correspondence \cite{Hecke}. 
The case of $G=GL_1$ encodes geometrically the vertex algebra of a free fermion, 
or in the language of integrable systems the KP flows together with the vertex
operator \cite{DJKM}.

In Section \ref{factorization}, for any smooth complex curve $X$ we define an algebro-geometric variant $\Gr_\D$
of the ad\`elic Grassmannian 
introduced by Wilson \cite{Wilson bispectral, Wilson CM}
in the study of the bispectral problem
and the rational solutions of KP (see Remark \ref{Ran space} for a discussion of the relationship of $\Gr_\D$ to Wilson's construction).  We then show  that 
$\Gr_\D$ is naturally a factorization space  in which 
the $GL_1$ Beilinson-Drinfeld Grassmannian embeds as a factorization subspace.  
This is a special case of a more general construction in Section \ref{factorization} of an ad\`elic
Grassmannian $\Gr_P$ associated to any $\D$-bundle $P$ on $X$; the case $P=\D^n$ leads to multicomponent
KP (and the $\Winfi({\mathfrak gl}_n)$ vertex algebra).

We show in Section \ref{W} that the ad\`elic Grassmannian realizes the $\Winfi$-algebra as
 the algebra of infinitesimal symmetries of $\D$-bundles:
\begin{thm}[See Theorem \ref{comparison thm}]
For any $c\in \C$, the chiral algebra $\Winfi^c({\mathfrak gl_n})$ on $X$
 at level $c$ associated to the $\Winfi({\mathfrak gl}_n)$-vertex algebra is isomorphic to the 
chiral algebra $\delta^c_{\D^n}$ of level $c$ delta functions along the unit section of the ad\`elic Grassmannian
$\Gr_{\D^n}$.
\end{thm}
\noindent
As we then explain in Section \ref{discussion}, the factorization structure of our ad\`elic Grassmannian thus
encodes both the infinitesimal symmetries and the Hecke
modifications of 
$\D$-bundles: that is, the $\Winfi$ vertex algebra
and the B\"acklund-Darboux 
transformations \cite{BHY1,BHY3}
of the KP hierarchy
(and, more generally,
the $\Winfi(\gl_n)$-algebras 
and B\"acklund-Darboux transformations  
of the multicomponent KP hierarchies).
In other words,
the factorization structure on the 
ad\`elic Grassmannian unites the vertex operators, B\"acklund transformations and
additional symmetries of the KP hierarchy \cite{vM} in a single geometric structure.
 The embedding of 
the $GL_n$ Beilinson-Drinfeld Grassmannian into the rank $n$ ad\`elic 
Grassmannian globalizes the inclusion of the Kac-Moody algebra
$\wh{\gl}_n$ into $\Winfi(\gl_n)$.   
Section \ref{discussion} also contains an extended explanation of the
interaction between our picture on the one hand and the free fermion
theory and the Krichever construction on the other, realizing $\Winfi$
orbits on the Sato Grassmannian, and thereby the Orlov-Schulman additional symmetries
of the KP hierarchy,
 via moduli of $\D$-bundles or cusp
line bundles.
We close with a discussion of the significance of the factorization space $\Gr_P$ for 
$\cW$-geometry.

For related studies of (and background on) the geometry of
$\D$-bundles we refer the reader to \cite{cusps,solitons,BGN}, which
originated as an attempt to understand geometrically and generalize
the relations between bispectrality, the ad\`elic Grassmannian, ideals
in the Weyl algebra and Calogero-Moser spaces first uncovered and explored in work of Wilson and Berest-Wilson \cite{Wilson bispectral,
  Wilson CM, BW automorphisms, BW ideals}. In \cite{cusps}, the
Cannings-Holland description of $\D$-modules on
curves and their cuspidal quotients is revisited from a more
algebro-geometric viewpoint (and the Morita equivalence of Smith-Stafford
generalized to arbitrary dimension).
In \cite{solitons}, two constructions of KP solutions from $\D$-bundles
are explained---one (which is most relevant to this paper) related to
line bundles on cusp curves and the other related to Calogero-Moser systems. 
(This description of moduli spaces of $\D$-bundles or ideals in $\D$ as Calogero-Moser
spaces is established for arbitrary curves in \cite{BGN}.)   As explained in
\cite{solitons}, in genus zero the two constructions are evidently
exchanged by the Fourier transform and thus, by \cite{BW automorphisms}, by Wilson's bispectral involution \cite{Wilson bispectral}. 
However, in higher genus the constructions are
completely independent.

\subsection{Acknowledgments}
The authors are grateful to Edward Frenkel for many helpful
discussions, in particular on the subject of $\cW$-geometry.  The
present paper grew out of discussions that began in 2002; during the
paper's gestational period, the authors have been supported by MSRI
postdoctoral fellowships, NSF postdoctoral fellowships, and NSF
grants.  DBZ is currently supported by NSF CAREER award DMS-0449830,
and TN by NSF award DMS-0500221.

\section{$\D$-bundles}\label{D-bundles}

In this section we review the definition and features of $\D$-bundles on a smooth complex algebraic curve $X$, following ideas of
Cannings-Holland \cite{CH ideals, CH cusps} as they are explained in \cite{cusps}. 
Let $\D=\D_X$ denote the sheaf of differential
operators on $X$ (which can be viewed as functions on the quantization of the
cotangent bundle $T^*X$). 

\begin{defn}
A {\em $\D$-bundle} $M$ on $X$ is a locally projective
coherent right $\D_X$-module.
\end{defn}

Equivalently, a $\D$-bundle is a {\em torsion-free} coherent right $\D_X$-module (since the sheaf
of algebras 
$\D_X$ locally has homological dimension one, \cite[Section 1.4(e)]{SS}).
Moreover, $\D_X$ possess a skew field of fractions
(see \cite[Section 2.3]{SS}), from
which it follows that any $\D$-bundle has a well-defined
rank. Thus $\D$-bundles can be viewed alternatively as vector bundles or as
torsion-free sheaves on the
quantized cotangent bundle of $X$.

Some obvious examples of $\D$-bundles are the locally free (or induced) rank $n$ 
$\D$-bundles, of the form $M=V\ot \cD$ for a rank $n$ vector bundle $V$ on $X$.
In general, however, $\D$-bundles of rank $1$ are not locally free (isomorphic to
$\cD$), 
but only {\em generically} free, behaving more like
the ideal sheaves of a collection of points on an algebraic surface. 
Note that every right ideal
in $\D_X$ is torsion-free, hence a $\D$-bundle of rank $1$ (conversely 
any rank one $\D$-bundle may locally
be embedded as a right ideal in $\D_X$).

For example, let $X=\aline$, so that $\D_X$ has as its ring of
global sections the Weyl algebra $\C\langle z,\del\rangle/\{\del
z-z\del=1\}$.  The right ideal $M_0$ generated by $z^2$ and $1-z\del$
agrees with $\D$ outside of $0$ but is not locally free near $z=0$.
Let 
\bd
W(M_0) = 
\{\theta(f) \,|\, \theta \in M_0, f\in \boldc[z]\}\subset \boldc[z].
\ed 
Then $W(M_0) = \boldc[z^2, z^3]$, the subring of $\boldc[z]$ generated by $z^2$ and $z^3$; this 
 is isomorphic to the coordinate ring of the cuspidal cubic
curve $y^2=x^3$.  Furthermore, $M_0 = \D(\boldc[z], W(M_0))$, the space of differential operators that take $\boldc[z]$ into
$W(M_0)$.  

\subsection{Grassmannian parametrization}
In general, to describe $\D$-modules it is convenient to extract linear algebra data
using the de Rham (or Riemann-Hilbert) functor, which (on right $\D$-modules)
takes $M$ to the sheaf of $\C$-vector spaces which is 
its quotient by all total derivatives, or formally
$$M\mapsto h(M):=M\ot_\D \cO_X.$$ Note that if $M=V\ot \D$ is induced,
we recover the underlying vector bundle $V=h(M)$ in this fashion (though only
as a sheaf of $\C$-vector spaces rather than as an $\cO_X$-module).  In the example of $M_0$ given above,
$h(M_0) = W(M_0)$.  

We can parametrize $\D$-bundles by choosing generic
trivializations: given a $\D$-bundle $M$ we can embed $M$ into
$n=\on{rk}M$ copies of rational differential operators,
$M\hookrightarrow \D^n(K_X)$, so that $M\ot K_X=\D^n(K_X)$.  The
module $M$ is then determined by the finite collection of points $x_i$
at which $M$ differs from $\D_X^n$, and choices of subspaces of $n$
copies of Laurent series ${\mathcal K}_{x_i}^n$ at $x_i$.  This is
succinctly explained in \cite[Section 2.1]{chiral}: $\D$-submodules of
a $\D$-module $N$ that are cosupported at a point $x\in X$ are in
canonical bijection (via the de Rham functor $h$) with subspaces of the
stalk of the de Rham cohomology $h(N)_x$ at $x$ that are open in a
natural topology. In our case, we obtain collections of linearly
compact open subspaces (or $c$-lattices, in the terminology of
\cite{chiral}) of $n$ copies of Laurent series ${\mathcal K}_{x_i}^n$
at $x_i$ with respect to the usual (Tate) topology, i.e., subspaces
commensurable with $n$ copies of Taylor series.  These subspaces
correspond to points of the ``thin Sato Grassmannian''
$\Gr(\cK_{x_i}^n)$ (not to be confused with the complementary original
or ``thick'' Sato Grassmannian, see Section \ref{discussion}).

\subsection{Cusps} $\D$-bundles can alternatively be described  
by torsion-free sheaves on cuspidal curves.  Namely, given a
$\D$-module $M$ with a generic trivialization (identification with
$\D^n$), we can find a subring (or rather subsheaf of rings)
$\cO_Y\subset \cO_X$ whose left action on $\D^n(K_X)$ (by right
$\D$-automorphisms) preserves $M$. Equivalently, the $\C$-subsheaf
$h(M)\subset h(\D(K_X))$ is preserved by a ring $\cO_Y\subset \cO_X$,
which differs from $\cO_X$ only at the finitely many singularities
$x_i$ of the embedding. Thus $h(M)$ defines a torsion-free sheaf $V_Y$
on the cuspidal curve $Y=\on{Spec}\cO_Y$, which has $X\to Y$ as a
bijective normalization.

In fact, as was shown by Smith and Stafford \cite{SS} and generalized to arbitrary
dimension 
in \cite{cusps}, passing from $X$ to such a {\em cuspidal quotients}
$Y$ doesn't change the category of $\D$-modules: the sheaves of rings $\D_X$ and
$\D_Y$ are Morita equivalent. Thus, given a
torsion-free sheaf $V_Y$ on $Y$ we can define a torsion-free
$\D_Y$-module $M_Y=V_Y\ot \D_Y$ by induction, and transport it to
obtain a $\D$-bundle $M$ on $X$. Moreover, this process of
``cusp-induction'' reverses the above procedure $M\mapsto V_Y$. 

As a result, we obtain a geometric reinterpretation of the linear algebra data
classifying $\D$-bundles. Every $\D$-bundle arises by cusp-induction
from a torsion-free sheaf on a cuspidal quotient $X\to Y$, but $Y$ is
not unique. We can consider $M$ as associated to smaller and smaller
subsheaves $\cO_{Y'} \subset\cO_Y\subset\cO_X$, introducing deeper and
deeper cusps in the curves $X\to Y\to Y'$. The collection of
$\D$-bundles on $X$ is obtained as a direct limit over these deepening
cusp curves of the collections of torsion-free sheaves. Under this
limit the geometry of $X$ ``evaporates'' and we are left with a
reinterpretation of the linear algebra data above, characterizing a
class of $\C$-sheaves on $X$ which arise from torsion-free sheaves on
some cusp quotient.

\section{Factorization}\label{factorization}
In this section we study the factorization space structure on the
symmetries of a $\D$-bundle. Factorization spaces were introduced in
\cite{chiral}; we refer the reader to \cite[Section 20]{book} for a
leisurely exposition. In Section \ref{factorization space}, we
establish a factorization space structure for $\D$-bundles on a curve.
As we explain in Section \ref{chiral algebra}, standard constructions
then produce a chiral algebra which acts infinitesimally simply
transitively near the unit section of our factorization space $\Gr_P$.
In Section \ref{levels}, we describe how to twist this chiral algebra
by the determinant line bundle to obtain a family of chiral algebras
at different levels.

\subsection{The Factorization Space $\Gr_P$}\label{factorization space}
Let $X$ denote a smooth complex curve. We usually fix a $\D$-bundle
$P$ in what follows.

Recall the definition of a {\em factorization space} or {\em
  factorization monoid} from \cite[Definition 2.2.1]{KV}.  We will
define a factorization space, the {\em $P$-ad\`elic Grassmannian}
$\Gr_P$ of $X$, as follows.  For each finite set $I$ we define
$\Gr_P^I(S)$, for a scheme $S$, as a set over $X^I(S)$, the set of
maps $S\xrightarrow{\phi} X^I$.  Such a map $\phi$ defines a divisor
$D\subset S\times X$.  Then $\Gr_P^I(S)$ is the set of pairs $(M,
\iota)$ consisting of:
\begin{enumerate}
\item An $S$-flat, locally finitely presented right $\D_{X\times
    S/S}$-module $M$, such that the restriction $M|_{X\times\{s\}}$ to
  each fiber is torsion-free (equivalently, locally projective).
\item An isomorphism $\iota: M|_{X\times S\smallsetminus D}
  \xrightarrow{\simeq} P|_{X\times S\smallsetminus D}$ such that the
  composite $M\rightarrow M|_{X\times S\smallsetminus D}
  \xrightarrow{\iota} P|_{X\times S\smallsetminus D}$ is injective
  with $S$-flat cokernel.
\end{enumerate}

\begin{remark}\label{Ran space}
As we have mentioned in the introduction, our ad\`elic Grassmannian is not the same as the one introduced by Wilson 
\cite{Wilson bispectral, Wilson CM}.  Indeed, our space differs from Wilson's in two significant respects.  One of these 
is that Wilson's space is closer to the colimit $\underset{\longrightarrow}{\lim}_I \Gr_P^I$ of spaces in our system.  This
colimit naturally lives 
 (see \cite{chiral} or \cite{book}) over the colimit $\underset{\longrightarrow}{\lim}_I X^I$ of finite products of the curve, which is
known as the {\em Ran space} ${\mathcal Ran}(X)$ of $X$.   The space ${\mathcal Ran}(X)$ 
is a contractible topological space that is not naturally by itself an object of algebraic geometry: it is better to work with it via the 
system $\{X^I\}_I$, and hence with the factorization space $\Gr_P^I$.  

The second difference is that Wilson identifies the spaces $h(M_1)$ and $h(M_2)$ associated to two $\D$-bundles $M_1$ and $M_2$
on ${\mathbf A}^1$ if they are identified by an element of $\boldc(x)^{\times}$, whereas we keep the two separate.  An alternative
to our approach, closer in spirit to Wilson's, would be to take the factorization subspace of $\Gr_P$ parametrizing those $M$ which ``have index zero with respect to $P$ at
every point of $X$.'' 
\end{remark}

Suppose $V\otimes\D$ is an induced $\D$-bundle.  
  Standard arguments using finite generation  show:
\begin{lemma}\label{finiteness lemma}
Suppose $(M,\iota)\in \Gr_{V\otimes\D}^I(S)$.  Then, locally on $S$, there exists an
integer $k$ such that 
$V(-kD)\otimes\D \subset M \subset V(kD)\otimes\D$.
\end{lemma}
\begin{comment}
  Here is a proof sketch of the lemma.  Observe that, since $M$ is
  locally finitely presented, the composite map $M\rightarrow
  M|_{X\times S \smallsetminus D} \xrightarrow{\iota} P|_{X\times
    S\smallsetminus D}$ locally factors through $P(nD) = V(nD)\otimes
  \D$ for some $n$.  Moreover, for each $n\gg 0$ we get a map
  $P(-nD)\rightarrow P(nD)/M$.  Since $P(nD)/M$ is a torsion module
  (it is a flat family of $\D$-modules supported set-theoretically on
  the divisor $D$), for which each section $s$ is killed by
  $\theo(-mD)$ for some $m\gg 0$, and the image of the map
  $P(-nD)\rightarrow P(nD)/M$ is finitely generated, there is some
  $m\gg 0$ for which the map $P(-mD)\rightarrow P(nD)/M$ is zero.
  Thus, locally on $S$, we have that for $n$ sufficiently large,
  $P(-nD)\subset M \subset P(nD)$.
\end{comment}
For any $\D$-bundle $P$, the space $\Gr_P$ is a factorization space:

\begin{prop}\label{factorization prop}
The functors $\Gr_P^I$ assemble into a factorization monoid over $X$ in the sense of
\cite[Definition 2.2.1]{KV}.
\end{prop}

\begin{proof}
  We need to show: for any set $I$, $\Gr_P^I$ is naturally an
  ind-scheme of ind-finite type, formally smooth over $X^I$, and comes
  equipped with a formally integrable connection over $X^I$.  The
  ``factorization'' properties themselves (parts (a) and (b) in the
  definition from \cite{KV}) are standard. Briefly, we need to see
  that the fiber of $\Gr_P^I$ over an $I$-tuple $x_I$ of ($S$-)points
  of $X$ depends only on the support of the tuple---not the
  multiplicities---and factorizes as a product for every disjoint union
  decomposition of the tuple. The first property is automatic for any
  functor defined in terms of data outside of $x_I$. The second
  follows from the fact that modifications at points $x_I$ are local,
  and so modifications away from disjoint points can be glued together
  (explicitly, this follows from the description of modifications by
  collections of points of Grassmannians, as in the previous section).

  Choose an induced $\D$-bundle $V\otimes \D$ and containments
  $V(-E)\otimes \D \subseteq P\subseteq V(E)\otimes \D$ (which is
  possible by Lemma \ref{finiteness lemma}).
  Then $P$ is cusp-induced from a cusp curve $Y$ with homeomorphism
  $X\rightarrow Y$ that fails to be an isomorphism only at the support
  of $E$.  Suppose $Y$ has a cusp located ``under'' a point $p$ in
  $X$.  If $z$ is a uniformizer at $p$, then for $n$ sufficiently
  large, $z^n$ lies in $\theo_Y$.  It follows that we can make sense
  of $\D$-module inclusions $\theo(-np)\otimes P\subset P \subset
  \theo(np)\otimes P$ (say, as subsheaves of $P|_{X\smallsetminus
    E}$).  Write $P(kD) = \theo(kD)\otimes P$ when this makes sense.
  We then give $\Gr_P$ the ind-structure coming from Lemma
  \ref{finiteness lemma}: we let $\Gr_P^I(n)$ denote the subset of
  $\Gr_P^I$ of pairs $(M,\iota)$ for which we have $P(-nD) \subset
  M\subset P(nD)$.
  
  There are now many ways to see that $\Gr_P^I(n)$ is a scheme of
  finite type.  For example, applying the de Rham functor to $M/P(-nD)
  \subset P(nD)/P(-nD)$ and applying Theorem 5.7 of \cite{cusps}, we
  see that choosing $M$ is equivalent to choosing a sheaf of vector
  subspaces of $h(P(nD)/P(-nD))$.  The functor of such choices is a
  closed subscheme of the relative Grassmannian of subspaces of
  $H^0\big(h(P(nD)/P(-nD))\big)$ (this is essentially the
  Cannings-Holland picture of the ad\`elic Grassmannian, as explained
  and used to great effect in \cite{Wilson CM}).
  
  It remains to show that $\Gr_P$ is formally smooth and to describe
  the formally integrable connection over $X^I$.  We first show that
  $\Gr_P^I \rightarrow X^I$ is formally smooth.  Let $(M,\iota)$ be an
  object of $\Gr_P^I(S)$ for a scheme $S$. We may assume that
  $P(-nD)\subset M \subset P(nD)$ and that $M$ is, in the terminology
  of \cite{cusps}, cusp-induced from a family $\cF$ of
  $\theo_Y$-modules, where $Y$ is a family of cusp curves over $S$
  given by a sheaf of sub-algebras $\theo_Y\subset \theo_{X\times S}$.
  We then have inclusions of $\theo_Y$-modules $h(P(-nD))\subset \cF
  \subset h(P(nD))$.  
  
  Now, given a nilpotent thickening $S\subset S'$
  and a divisor $D'$ on $S'\times X$ that is a nilpotent thickening of
  the divisor $D$ to which $(M,\iota)$ corresponds, we need to extend
  $(M,\iota)$ to a pair $(M',\iota')$ over $X\times S'$.  To do this,
  we first extend $Y$ to a family $Y'$ of cusp curves over $S'$ ``with
  cusps determined by $D'$,'' i.e. cusps whose depths are determined
  by $D'$.  It then suffices to deform the map $h(P(nD))\rightarrow
  h(P(nD))/\cF$: the kernel of the cusp-induction of the deformed map
  will give $M'$, and the inclusion in $h(P(nD'))\otimes \D$ will give
  $\iota'$.  But, by construction, as an $\theo_Y$-module
  $h(P(nD))/\cF$ is a direct sum of skyscrapers: more precisely, it is
  isomorphic to the direct image to $Y$ of $\oplus_{i\in I}
  \theo_{p_i}$, where $p_i$ is the $i$th section of $X\times
  S\rightarrow S$ determined by the map $S\rightarrow X^I$.  So, as
  our deformation of $h(P(nD))/\cF$ we can take $\oplus_{i\in I}
  \theo_{p_i}$ as an $\theo_{Y'}$-module, and then deformations of
  this sum of skyscrapers, as well as the map from $h(P(nD'))$,
  certainly exist.  This proves the existence of $(M',\iota')$.  Hence
  $\Gr_P^I\rightarrow X^I$ is formally smooth.  (Note that an
  alternative proof of this assertion uses the formally transitive
  action of a Lie algebra of matrix differential operators, as in the
  next section.)

  Finally, we describe the formally integrable connection over
  $X^I$---note, however, that the existence of such a connection follows formally (as in
  \cite[20.3.8]{book}) from the existence of a unit for the
  factorization space.  This follows \cite[Section 5.2]{Ga}.
  Namely, given an Artinian scheme $A$, a pair $(M,\iota)$
  parametrized by a scheme $S$, and a map $S\times A \rightarrow X^I$
  determining a divisor $D_A$, we pull $M$ back along the projection
  $X\times S \times A \rightarrow X\times S$.  Since $X\times S\times
  A\smallsetminus D_A = (X\times S\smallsetminus D)\times A$, this
  pullback comes equipped with an isomorphism to $P$ over $X\times
  S\times A\smallsetminus D_A$, as desired.  This canonical lift of
  infinitesimal extensions gives our connection.
\end{proof}

Suppose $P=V\otimes \D$ is an induced $\D$-bundle.  Let $\on{Gr}_V^I$ denote the
Beilinson-Drinfeld Grassmannian associated
to $V$: over a divisor $D\subset X$, this parametrizes vector bundles $W$ equipped
with an isomorphism 
$W|_{X\smallsetminus D}\cong V|_{X\smallsetminus D}$.  Such an isomorphism induces a
$\D$-bundle isomorphism
$W\otimes \D|_{X\smallsetminus D}\cong P|_{X\smallsetminus D}$, thus giving an
embedding of unital factorization spaces:
\begin{equation}\label{BD in Grad}
\on{Gr}_V^I\hookrightarrow \Gr_P^I.
\end{equation}

We note that the relative tangent space of $\Gr_P^I\rightarrow X^I$ at $(M,\iota)$
is given by
\begin{equation}\label{tangent space}
T_{(M,\iota)} (\Gr_P^I/X^I) = \Hom_\D(M, P(\infty\cdot D)/M)
\end{equation}  
by the same analysis as one uses for a Quot-scheme.  Given a divisor
$D\subset X$, let $\on{supp}(D)$ denote the support of the divisor $D$
and $\widehat{X}_D$ the formal completion of $X$ along $D$.  For a
quasi-coherent sheaf $\cF$, let $\cF_\eta$ denote the direct sum of
stalks at points of $\on{supp}(D)$.
\begin{lemma}\label{tangent spaces}
  There is a natural surjection: \bd \underline{\End}(P(\infty\cdot
  D))_\eta\twoheadrightarrow T_{(M, \iota)}(\Gr_\D^I/X^I).  \ed If, in
  addition, $P=\D^n$, this yields a surjection: \bd H^0(\D(\infty\cdot
  D)|_{\widehat{X}_D})\otimes {\mathfrak gl}_n\twoheadrightarrow
  T_{(M, \iota)}(\Gr_\D^I/X^I).  \ed In particular, if $D$ consists of
  a single point $x$, we get a surjection \bd \D(\cK_x)\otimes
  {\mathfrak gl}_n\twoheadrightarrow T_{(M,
    \iota)}(\Gr_\D^{\{\ast\}}/X).  \ed
\end{lemma}
\begin{proof}
  We use the short exact sequence: \bd 0\rightarrow M\rightarrow
  P(\infty\cdot D)\rightarrow P(\infty\cdot D)/M\rightarrow 0.  \ed
  Since $M$ and $P(\infty\cdot D)$ are locally projective, the sheaf
  Ext group $\underline{\Ext}^1_\D\big(P(\infty \cdot D),M\big)$
  vanishes, and we get a surjective map of sheaves \bd
  \uHom_\D\big(P(\infty\cdot D), P(\infty\cdot D)\big)\rightarrow
  \uHom\big(P(\infty\cdot D), P(\infty\cdot D)/M\big).  \ed Since
  $\uHom\big(P(\infty\cdot D), P(\infty\cdot D)/M\big)$ is supported
  along $\on{supp}(D)$, the conclusions follow.
\end{proof}

\subsection{Chiral Algebra}\label{chiral algebra}
The $\D$-bundle $P$, equipped with its canonical isomorphism $id$ with
$P$ over $X\smallsetminus D$ for any $D$, defines a ``unit'' section
$\unit^I: X^I\to\Gr_P^I$, compatible with the factorization
isomorphisms and preserved by the relative connection.  Let $\delta_P$
denote the $\Oo$-push-forward of the right $\D$-module
$\unit^1_!\omega_X$ on $\Gr_P^X=\Gr_P^{\{1\}}$, and $\delta_P^I$ the
push-forward of $\unit^I_!\omega_{X^I}$ from $\Gr_P^I$. The sheaves
$\delta_P^I$ are right $\D$-modules on $X^I$, and we let
${}^\ell\delta_P^I$ denote the corresponding left $\D$-modules.

\begin{corollary} The left $\D$-modules ${}^\ell\delta_P^I$ 
form a factorization algebra on $X$. In particular $\delta_P$ is a
chiral algebra.
\end{corollary}
\begin{proof}
This exactly follows Section 5.3.1 of \cite{Ga}.
\end{proof}
\begin{remark}
  We will identify this chiral algebra with a $\cW$-algebra below in
  the case $P = \D^n$.  It would be interesting to investigate the
  structure of $\delta_P$ further in the case that $P$ is not locally
  free.
\end{remark}
\begin{prop} 
  Suppose $P = \D^n$.  The Lie algebra $\D(\K_x)\otimes {\mathfrak
    gl}_n$ acts continuously and formally transitively on the
  Grassmannian fiber $\Gr^{\{\ast\}}_{\D^n}(x)$, inducing an
  isomorphism $T_{\D^n}\Gr^{\{\ast\}}_{\D^n}(x)=\D(\K_x)/\D(\Oo_x)
  \otimes {\mathfrak gl}_n$.
\end{prop}

\begin{proof}
In the Grassmannian description of the fiber $\Gr_{\D^n}(x)=\Gr(\K_x^n)$, the action of
$\D(\K_x)\otimes {\mathfrak gl}_n$ comes from its defining action by continuous
endomorphisms of $\K_x^n$.
The tangent space to the Grassmannian $\Gr(\K_x^n)$ at $\Oo_x^n$ is identified with
$$\Hom_{\on{cont}}(\Oo_x^n,\K_x^n/\Oo_x^n)=\D(\K_x)/\D(\Oo_x)\otimes {\mathfrak
gl}_n,$$ since a
continuous homomorphism factors through a map $(\Oo_x/\bm_x^k)^n\to
(\bm_x^{-k}/\Oo_x)^n$ for $k$ sufficiently large, and all such
homomorphisms may be realized by differential operators (or
alternatively since
$\Hom_{\on{cont}}(\Oo_x^n,\K_x^n/\Oo_x^n)=\Hom_\D(\D^n,\K_x^n/\Oo_x^n\ot\D)$,
since a homomorphism on either side is automatically $\OY$-linear for
a deep enough cusp $Y$). The formal transitivity at other points follows similarly.
\end{proof}

It follows that $\D(\K_x)\otimes {\mathfrak gl}_n$ also acts on the space of delta
functions
$\delta_{\D^n}(x)$---we denote this action by
$$\mathsf{act}_x:\big(\D(\K_x)\otimes{\mathfrak gl}_n\big)
\ot\delta_{\D^n}(x)\to\delta_{\D^n}(x).$$
This action induces an isomorphism of $\D(\K_x)\otimes{\mathfrak gl}_n$-modules
\begin{equation}\label{delta as module}
\delta_{\D^n}|_x\cong U\big(\D(\K_x)\otimes {\mathfrak
gl}_n\big)\ot_{U\big(\D(\Oo_x)\otimes {\mathfrak gl_n}\big)}\C.
\end{equation}

\subsection{Levels}\label{levels}
Let $M$ be a coherent $\D_{X\times S/S}$-module on $X\times
S\stackrel{\pi}{\to}S$. We define a line bundle $\det(M)$ on $S$ as
follows. Consider the bounded complex
$R\pi_{\cdot}\on{DR}^{\cdot}(M)=R\pi_{\cdot}(M\dR\Oo)$ of coherent
$\Oo_S$-modules, and form its determinant line bundle
$$\det(M)=\det(R\pi_\cdot M\dR\Oo).$$
Note that if $S$ is smooth and
$M$ is in fact a right (absolute) $\D_{X\times S}$-module, then
$R\pi_*M=R\pi_\cdot M\dR\Oo\in\rmod-\D_S$ is the right $\D$-module
push-forward of $M$. Thus the determinant of de Rham cohomology line
bundle $\det(M)$ is the $\D$-module analogue of the determinant of
cohomology line bundles on moduli stacks of bundles.

Suppose that $M$ is a cusp-induced $\D$-module, so that $M\cong \Mbar\otY\DYX$
for a cuspidal quotient $\wb{\pi}:Y\to S$ of $X\times S\to S$ over $S$.
Then we have $\on{DR}^{\cdot}M=h(M)=\Mbar$, and $\det(M)$
is the determinant of cohomology
$$\det(M)=\det(R\wb{\pi}_*\Mbar)$$
of the $\Oo$-module push-forward of
$\Mbar$ to $S$.  It follows that the determinant line bundle,
restricted to a fiber $\Gr_\D^{\{\ast\}}(x)$, is naturally identified
with the ``Pl\"ucker'' determinant line bundle on the Grassmannian
$\Gr(\K_x^n)$.  The following proposition follows from the description
of $\det$ on $\Gr_{\D^n}^I|_{\xarr}$ as a tensor product of local
factors at the points of $\xarr$ (for an algebraic treatment of
factorization in the closely related context of $\epsilon$-factors,
see \cite{BBE}):

\begin{prop} The line bundle $\det$ on $\Gr_{P}$ has a natural factorization structure
  over the factorization structure of $\Gr_{P}$.
\end{prop}

We can now define a new chiral algebra $\delta_P^c$ for every $c\in
\C$ as level $c$ delta functions along the unit section on $\Gr_P$.
More precisely, the factorization line bundle $\det$ has a lift
$\on{unit}^{\det}$ of the unit section of $\Gr_P$: in other words,
$\det$ is canonically trivialized along $\unit$.  As a consequence,
there is a natural direct image functor from $\D_{X^I}$-modules to
{\em $\det^{\otimes c}$-twisted} $\D$-modules on $\Gr_P^I$.  It
follows that we may take $\omega_{X^I}$, push it forward to $\Gr_P^I$
as a $\det^{\otimes c}$-twisted $\D$-module, and take its
$\theo$-module direct image to $X^I$ to obtain a new chiral algebra,
denoted by $\delta_P^c$. (Alternatively we can describe $\delta_P^c$
as the $\cO$-module restriction of the sheaf of $\det^c$-twisted
differential operators to the unit section.)

\section{$\Winfi$ and $\Winfi({\mathfrak gl}_n)$}\label{W}
In this section, we explain how to identify the chiral algebra
$\delta_{\D}^c$ or, more generally, $\delta_{\D^n}^c$: it is the
chiral algebra associated to the vertex algebra $\Winfi$
(respectively, $\Winfi({\mathfrak gl}_n)$) at level $c$. We refer to
\cite{Kac} for a discussion of $\Winfi$ and \cite{vdL} for
$\Winfi(\gl_n)$.  We begin in Section \ref{introducing winfi} by
reviewing the definition of the chiral algebras $\Winfi$ and, more
generally, $\Winfi({\mathfrak gl}_n)$.  In Section \ref{comparison} we
then identify them with $\delta_{\D}^c$ and $\delta_{\D^n}^c$.

\subsection{Introducing $\Winfi$}\label{introducing winfi}
We recall (from \cite{chiral}, 2.5 and \cite{BS}) the construction of
the Lie* algebra (or conformal algebra \cite{Kac}) $\gl(\D)={\mc
  End}^*\D$.  As a right $\D$-module, $\gl(\D)$ is the induced
$\D$-module $\D\ot\D$. We consider the sheaf $\D$ as a quasi-coherent
$\Oo$-module equipped with a Lie algebra structure given by
differential morphisms. (It forms the Lie algebra of right $\D$-module
endomorphisms of $\D$.) It follows that the induced $\D$-module has a
natural Lie* algebra structure, as explained in \cite{chiral} or
\cite[19,1,7]{book}.  We let $\Winfi=U_{ch}\gl(\D)$ denote the
universal enveloping chiral algebra of $\gl(\D)$. It follows that the
fibers of $\Winfi$ are identified with the vacuum representation of
$\D(\K_x)$,
\begin{equation}\label{Winfi as module}
\Winfi(x)\cong U\D(\K_x)\ot_{U\D(\Oo_x)}\C.
\end{equation}

The Lie* algebra $\gl(\D)$ has a central extension $\wh{\gl}(\D)$ which
may be described as follows.  Consider the short exact sequence
$$0\to \Oo\boxtimes \omega_X\to j_*j^*\Oo\boxtimes\omega_X\to \D\to
0,$$
where $\D$ is considered as $\Oo$-bimodule. Restriction to the
diagonal defines a morphism $\Oo\boxtimes\omega_X\to
\Delta_*\omega_X$, and we take the push-out exact sequence (and push
forward to $X$), obtaining a central extension
$$0\to \omega_X\to \wh{\D}\to \D\to 0.$$ This $\omega_X$-extension of
$\D$ corresponds to an $\omega_X$-central extension of the Lie*
algebra $\gl(\D)$. It is also shown in \cite{chiral,BS}
that the action of $\D(\K_x)$ on $\K_x$ gives rise to a dense embedding of
$\D(\K_x)$ in
$\gl(\K_x)$, the Tate endomorphisms of $\K_x$, and that the Tate central extension
of $\gl(\K_x)$ restricts to the above central extension $\wh{\D}(\K_x)$. 
The following lemma is an immediate consequence: 

\begin{lemma} The action of $\D(\K_x)$ on $\Gr(\K_x)$ lifts to an action of the
central extension $\wh{\D}(\K_x)$ on $\det$ with level one.
\end{lemma}

We denote $\Winfi^c$ the corresponding chiral enveloping algebra at
level $c$.
Repeating the construction with $\D\otimes {\mathfrak gl}_n$ in place of $\D$ leads
to the chiral algebra $\cW^c_{\infty}({\mathfrak gl}_n)$.

\subsection{Comparison of $\delta_\D$ and $\Winfi$}\label{comparison}
In this section we identify the chiral algebra $\delta_{\D^n}^c$ associated
to the factorization space $\Gr_{\D^n}$ with the $\mc W$-algebra
$\cW^c_{\infty}({\mathfrak gl}_n)$:

\begin{thm}\label{comparison thm}
There is a natural isomorphism of chiral algebras $\cW_\infty({\mathfrak gl}_n)\cong
\delta_{\D^n}$.  Moreover for any $c\in \C$, this isomorphism lifts to an
isomorphism of the chiral algebra $\cW^c_{\infty}({\mathfrak gl_n})$ at level $c$
with the
chiral algebra $\delta^c_{\D^n}$ of level $c$ delta functions on
$\Gr_{\D^n}$.
\end{thm}

We will explain two proofs of this theorem below.  First, in Section
\ref{chiral hopf}, we explain that the theorem is a special case of a
very general construction of chiral algebras from the unit section of
a unital factorization space.  Because this general construction seems to be a 
folk theorem that does not appear in the literature, we also
explain in Section \ref{special proof} a more concrete proof that
follows closely the approach to Kac-Moody algebras explained in
\cite{Ga}.

\subsubsection{Chiral Hopf Algebras and Factorization Spaces}\label{chiral hopf}
We begin with some generalities.  Let $\{S^I\rightarrow X^I\}_I$ be a
unital factorization space.  Recall that, in particular, each $S^I$ is
a formally smooth ind-scheme of ind-finite type over $X^I$.  Write $S=
S^{\{\ast\}}\xrightarrow{r} X$.  Pulling back the tangent sheaf
$\unit^* T_S$ along the unit section gives a $\D_X$-module (using the
flat connection on $S/X$) which we denote by $L$.  The factorization
structure on $S$ equips $L$ with a structure of Lie*-algebra on $X$.
We assume that this Lie*-algebra is an ind-$\D$-vector bundle---that
is, that it is a colimit of $\D$-vector bundles $L_j$ and the
Lie*-algebra structure is compatible with this realization (in the
standard sense that $L_j$ and $L_k$ multiply to $L_{j+k}$).  It then
follows from \cite[Lemma~2.5.7]{chiral} that the dual $\D$-module
$L^\vee$ is a pro-Lie$^!$-coalgebra, \cite[Section~2.5.7]{chiral}.

Write $\wh{S}$ for the formal completion of $S$ along the unit
section; then $\wh{S}$ is again a unital factorization space over $X$.
This space can be completely reconstructed from the Lie*-algebra $L$:
this is essentially just the factorization analog of the fact that a
smooth formal group scheme can be reconstructed from its Lie algebra.
In the factorization setting, this works as follows.  To the
Lie*-algebra one can associate a chiral algebra, the {\em chiral
  enveloping algebra} $U^{ch}(L)$ of $L$ \cite[Section~3.7]{chiral}.
The chiral enveloping algebra is naturally a cocommutative Hopf chiral
algebra (defined in \cite[Section~3.4.16]{chiral}.  The dual
$\D$-module $U^{ch}(L)^\vee$ is then a commutative pro-Hopf chiral
algebra, and, by the discussion in \cite[Section~3.4.17]{chiral}, its
formal spectrum gives a factorization space $\wh{S}(L) =
\{\wh{S}(L)^I\rightarrow X^I\}$.  Under the hypothesis that $L$ is an
ind-$\D$-vector bundle, moreover, $\wh{S}(L)$ is isomorphic to
$\wh{S}$, the formal completion of our original factorization space
along the unit section.

As an immediate consequence, one has:
\begin{thm}\label{folk thm}
  Suppose that $\{S^I\rightarrow X^I\}_I$ is a unital factorization
  space and that the associated Lie*-algebra $L$ is an ind-$\D$-vector
  bundle.  Then the delta-function chiral algebra $\delta = r_*
  \unit_! \omega_X$ is isomorphic to $U^{ch}(L)$.
\end{thm}
\noindent
There is also an extension of the theorem to the twisted chiral algebra $\delta^c$
associated to a factorization line bundle $\cL^I$ on 
$S^I$.  We leave it to the reader to formulate the analog of Theorem \ref{folk thm}
in this case.

Theorem \ref{comparison thm} is an immediate corollary once we identify $L$ with 
${\mathfrak gl}(\D^n)$; this is a consequence of Proposition \ref{fact action}.

\subsubsection{Direct Proof of Theorem \ref{comparison thm}}\label{special proof}
In this section we give a somewhat
 different, more direct and concrete, identification of the chiral algebra
$\delta_{\D^n}$ associated
to the factorization space $\Gr_{\D^n}$ with the $\mc W$-algebra $\Winfi({\mathfrak
gl}_n)$.
For simplicity of exposition, we restrict to the case $n=1$ and $c=0$, in other
words to $\Winfi$.
The proof is modeled directly on the identification of the Kac-Moody
vertex algebra with the factorization algebra of delta functions on
the affine Grassmannian due to Beilinson and Drinfeld explained in \cite{Ga}. 
The result is a formal consequence of the
construction of an extension of the action of $\D(\K_x)$ on
$\Gr_\D(x)$ to a ``factorization action'' of the Lie* algebra
$\gl(\D)$ on the factorization space $\Gr_\D$ (Proposition
\ref{fact action} below).

We assume that $X$ is affine in the construction, though the final
result will not depend on this assumption.  Recall (\cite{chiral} 3.7,
\cite{Ga}) that to any right $\D$-module $L$ on $X$ we can assign a
collection $\wt{L}^\ell$ of left $\D$-modules on $X$, with certain
factorization isomorphisms.  Consider the open subset
$j^{(I)}:U\subset X^I\times X$ of pairs $(\{x_i\},x)$ where $x\neq
x_i$ for any $i\in I$. We let $\wt{L}^I$ denote the relative de Rham
cohomology sheaf $\wt{L}^I=p_{X^I\cdot}(\on{id}\boxtimes
h)(j_*^{(I)}j^{(I)*} \Oo_{X^I}\boxt L)$ on $X^I$, whose fiber at
$\{x_i\}$ is the de Rham cohomology of $L$ on $X$ with poles allowed at
the $x_i$. The sheaf inherits a left $\D$-module structure from that
on $\Oo_{X^I}$. The sheaves $\wt{L}^I$ satisfy several good
factorization properties with respect to maps of sets $J\to I$
(\cite{chiral}), in particular for $J\twoheadrightarrow I$ we have
canonical isomorphisms $\Delta_I^!\wt{L}^J\cong \wt{L}^I$.

The sheaf $\wt{L}^I$ contains as a $\D$-submodule the de Rham
cohomology of $L$ with no poles allowed,
$\wb{L}^I=p_{X^I*}(\Oo_{X^I}\boxt L)$.  We define $L^I$ as the
quotient sheaf, and note that this sheaf is local in $X$, in the sense
that the fiber at $\{x_i\}$ only depends on $L$ in the neighborhood of
the $x_i$. In particular $$L^{\{1\}}=p_{X*}(j_*j^* (\Oo_X\boxt
L)/\Oo_X\boxt L)=L^\ell$$ is the left $\D$-module version of $L$. When
$L$ is a Lie* algebra, then the sheaf $\wt{L}^I$ is a Lie algebra in
the tensor category of left $\D$-modules, and $\wb{L}^I$ is a Lie
sub-algebra.  The sheaf $U\wt{L}^I\ot_{U\wb{L}^I}\C$ of induced
(vacuum) representations then forms a factorization algebra,
corresponding to the chiral
enveloping algebra $U^{ch}(L)$ of $L$.

In our case, the Lie* algebra $L=\gl(\D)$ is an induced $\D$-module,
with de Rham cohomology $h(\gl(\D))=\D$. It follows that the sheaves
$\wt{\gl}(\D)^I$ may be identified (as Lie algebras in left
$\D$-modules) with the usual (sheaf or $\Oo$-module) push-forward
$$\wt{\gl}(\D)^I=p_{X^I\cdot}(j_*^{(I)}j^{(I)*}\Oo_{X^I}\boxt
\D)=p_{X^I*}(\D_{X^I\times X/X^I}(*\xarr),$$ i.e. with differential
operators with arbitrary poles along the universal divisor on
$X^I\times X$ over $X^I$. The subsheaf $\wb{\gl}(\D)^I$ then consists
of global differential operators on $X$. The quotient sheaf has fiber
at $x$ the space $\D(\K_x)/\D(\Oo_x)$ of delta-functions at $x$, and
has fiber at $\{x_i\}$ the global sections of the $\D$-module on $X$
which consists of delta functions at each of the points $x_i$, without
multiplicities.  
%%The quotient sheaf is flat over $X^I$.  

We now describe the ``factorization action'' of $\wt{\gl}(\D)$ on
$\Gr_\D$, which is a multipoint generalization of the formally
transitive action of $\D(\K_x)$ on $\Gr_\D(x)$ and resulting
description $T_{\D}\Gr_\D|_x=\D(\K_x)/D(\Oo_x)$ of the tangent
space.

\begin{prop} \label{fact action}
\mbox{}
\begin{enumerate}
\item The Lie algebra $\wt{\gl}(\D)^I$ in left $\D$-modules over
$X^I$ acts on $\Gr_\D^I$, compatibly with connections and
factorization structures, and with stabilizer at the unit section the
Lie sub-algebra $\wb{\gl}(D)^I$.
\item For $x\in X$, the Lie algebra $\wt{\gl}(\D)(x)$ is dense in
$\D(\K_x)$, and its action on $\Gr_\D(x)$ is the restriction of the 
formally transitive action of $\D(\K_x)$.
\end{enumerate}
\end{prop}

\begin{proof}
Consider the sheaf of Lie algebras
$\wt{\gl}(\D)^I=p_{X^I*}\D(*\xarr)$ of differential operators
with poles along the universal divisor $\xarr$. It acts by
infinitesimal automorphisms on the sheaf $\D(*\xarr)$, and hence on
its functor of submodules, preserving the sub-functor of submodules
cosupported on $\xarr$.  This is our desired action on
$\Gr_\D^I$. The action is compatible with decompositions of the
divisor and forgetting of multiplicities, hence with the factorization
structure. The formal transitivity follows from Lemma \ref{tangent spaces}.
\end{proof}

We are now ready to prove Theorem \ref{comparison thm}.

\begin{proof}[Proof of Theorem \ref{comparison thm}]
\newcommand{\act}{\mathsf{act}}
We begin by constructing a natural map 
$$\act: j_*j^*(\gl(\D)\boxt\delta_\D)\to \Delta_!\delta_\D$$ satisfying
\begin{enumerate}
\item $\gl(\D)$ acts on the chiral algebra $\delta_\D$ by derivations.
\item The resulting action on $\delta_\D(x)$ of the completed de Rham
cohomology $\wh{h}_x(\gl(\D))$ coincides with the action
$\D(\K_x)\ot \delta_\D(x)\to\delta_\D(x)$ induced by the action
$\act_x$ of $\D(\K_x)$ on $\Gr_\D(x)$.
\end{enumerate}
To construct the map $\act$, it is equivalent to define the map
obtained by taking de Rham cohomology along the first factor, that is a
map $(j_*j^*(\D\boxt\Oo))\boxt \delta_\D\to \Delta_*\delta_\D.$ This
map is supported set-theoretically on the diagonal, and so is
equivalent to a continuous $\D$-module action of the completion
$p_{1*}(j_*j^*(\D\boxt\Oo))\wh{}$ along the diagonal on
$\delta_\D$. Note that
$$\wt{\gl}(\D)^X=p_{1*}(j_*j^*(\D\boxt\Oo))\subset
p_{1*}(j_*j^*(\D\boxt\Oo))\wh{}$$ is dense since $X$ is affine, and so
it suffices to define a continuous $\D$-module action of
$\wt{\gl}(\D)^X$ on $\delta_\D$. Such an action is provided by 
Proposition \ref{fact action}. The statement that this action is by derivations
of the chiral algebra $\delta_\D$ follows from the compatibility of
the action with factorization. More precisely, the action of
$\wt{\gl}(\D)^X$ on $\delta_\D$ is the restriction to the diagonal of
the action of $\wt{\gl}(\D)^I$ on $\delta_\D^I$ for $I=\{1,2\}$, and
the chiral bracket on $\delta_\D$ is simply the gluing data for
$\delta_\D^I$ along the diagonal, so that we have a commutative
diagram
\begin{equation*}
\begin{array}{ccccc}
j_*j^*(\wt{\gl}(\D)^{\{1,2\}}\ot \delta_\D^{\{1,2\}})&=& 
\wt{\gl}(\D)^{\{1,2\}}\ot j_*j^*(\delta_\D\boxt\delta_\D)
&\longrightarrow& j_*j^*(\delta_\D\boxt\delta_\D)\\
\downarrow&&\downarrow&&\downarrow\\
\wt{\gl}(\D)^{\{1,2\}}\ot\Delta_!\delta_\D&=&\Delta_!(\wt{\gl}(\D)^X\ot\delta_\D)&
\longrightarrow& \Delta_! \delta_\D.
\end{array}
\end{equation*}
The compatibility with the one-point action follows from the
compatibility of the action of $\wt{\gl}(\D)^X(x)$.

Next we construct a map of $\gl(\D)$ in $\delta_\D$ which on fibers
is the embedding of the tangent space
$i_x^!\gl(\D)=\D(\K_x)/\D(\Oo_x)$ into delta-functions on
$\Gr_\D(x)$. This is simply the image of the action map $\D(\K_x)\cdot
1\subset \delta_\D(x)$ on the unit, so the families version is the map
$$j_*j^*(\gl(\D)\boxt\omega_X)\xrightarrow{\act(\unit)}
\Delta_!\delta_\D$$ given by the action $\act$ on the unit, which
factors through a map $\gl(\D)\to\delta_\D$.  It follows from the
derivation property of $\act$ that this is a map of chiral modules
over $\gl(\D)$. Moreover we claim that the restriction of the chiral
bracket of $\delta_\D$ to $j_*j^*(\gl(\D)\boxt\delta_\D)$ is the same
as the action map $\act$. This follows from the derivation
property of $\act$ and the unit axiom, by comparing the two ways
of multiplying $\gl(\D)\boxt \omega_X \boxt \delta_\D$.

It now follows that the map $\gl(\D)\to\delta_\D$ lifts to a
homomorphism of chiral algebras and $\gl(\D)$-modules
$\Winfi\to\delta_\D$, which on fibers induces the isomorphism
$\Winfi(x)\cong\delta_\D(x)$ arising from the actions of $\D(\K_x)$,
\eqref{delta as module} and \eqref{Winfi as module}.  The theorem at
level zero follows.  Finally, the theorem at arbitrary level follows
from the statement concerning the lifting of $\D(\K_x)$ to the
determinant line bundle on $\Gr_\D(x)$ at level one.
\end{proof}

\section{$\Winfi$-symmetry and integrable systems}\label{discussion}
In this section, we explain how the factorization space $\Gr_{\D^n}$ unifies
geometric features of the free fermion CFT, the KP hierarchy, and the 
geometry of moduli of curves and bundles.  We begin by reviewing the geometry of the
Krichever construction (Section \ref{krichever}).  
We then explain, from the point of view of the geometry of $\D$-bundles, the
additional $\Winfi$-symmetry of these systems (Section
\ref{W-orbits}).  The infinitesimal $\Winfi$-symmetry is only part of the full
symmetry 
encoded in the factorization space $\Gr_{\D^n}$, however.  In Section
\ref{backlund-darboux}, we explain the relationship between 
global symmetries of bundles---the Hecke modifications---and of solitons---the
B\"acklund-Darboux transformations.  Finally, we explore (Section
\ref{W-geometry}) the significance of 
$\Gr_{\D^n}$ for the interesting and still-mysterious
subject of $\cW$-geometry.

\subsection{The Krichever construction}\label{krichever}
Let us begin by briefly reviewing the Krichever construction of
solutions of KP in the formalism of the Sato Grassmannian following
\cite{KNTY}---see also \cite{DJM,SW,Mu}.

Let $\cK=\C((z))$ denote the field of Laurent series, and
$\cO=\C[[z]]$ the ring of Taylor series.  The Sato Grassmannian $\GR$
\cite{Sa} parametrizes subspaces of $\cK$ which are complementary to
subspaces commensurable with $\cO$ (see \cite{AMP,solitons} for
algebraic constructions of $\GR$. Note that this Grassmannian, which is a scheme,
is quite different from the ind-scheme $\Gr(\cK)$, the thin Grassmannian,
parametrizing subspaces commensurable with $\cO$.)   

The Lie algebras $\cK$ and $\Der
\cK=\C((z))\del_z$ of Laurent series (with zero bracket) and Laurent
vector fields act on $\GR$ through their action (by multiplication and
derivation) on $\cK$ itself. The KP hierarchy appears as the
infinitely many commuting flows coming from the action of $\cK$ (or
more precisely of its sub-algebra $\C[z\inv]$) on the Grassmannian (or
more precisely on its big cell, in natural coordinates).

Given a smooth projective complex curve $(X,x)$ and a line bundle
$\cL$ on it, together with a formal coordinate and a formal
trivialization of $\cL$ at $x$, Krichever's construction produces a
point of the Sato Grassmannian, and hence a solution of the KP
hierarchy which may be expressed using theta functions.  Namely, we
consider the space $\cL(X\sm x)$ of sections on the punctured curve,
embedded in $\cK$ using the coordinate and trivialization. Note that
for this construction we only require $X$ to be smooth at $x$, and can
consider on an equal footing rank one torsion-free sheaves on singular
curves with a marked smooth point.  The construction can also be
described as specifying the conformal blocks (or correlation
functions) for the theory of a free fermion or a free boson on $X$
twisted by $\cL$ \cite{KNTY,Ueno,book}.

The KP flows (action of $\cK$) vary the line bundle $\cL$ via
translation in the Picard group of $X$ (the positive half $\cO$ merely
changing the trivialization). This action is in fact infinitesimally
transitive on the moduli space $\wh{Pic}(X,x)$ of line bundles on $X$
with a formal trivialization at $x$, thereby giving a formal
uniformization of the Picard group $Pic(X)$.  Likewise, the action of
$\on{Der}\cK$ (for fixed line bundle $\cL=\cO_X$) or the action of the
semi-direct product $\cK\oplus \Der \cK=\cD_{\leq 1}(\cK)$ deforms the
pointed curve $(X,x)$ (with the Taylor vector fields
$\on{Der}\cO=\C[[z]]\del_z$ merely changing the coordinate and moving
the point along $X$).  Again, these actions are infinitesimally
transitive on the moduli space $\wh{\fM}_{g,1}$ of pointed curves with
formal coordinate or of the $\wh{Pic}(X,x)$-bundle $\wh{Pic}_{g,1}$ of
all Krichever data, and give formal uniformizations of the moduli
spaces $\fM_{g,1}$ and $Pic_{g,1}$: see \cite{Ko,ADKP,BS,TUY}, or
\cite{book} for a detailed exposition. We can thus consider the moduli
spaces $\wh{Pic}(X,x)$, $\wh{\fM}_{g,1}$ and $\wh{Pic}_{g,1}$ as
global (integrated) versions of homogeneous spaces for the
corresponding Lie algebras.  The same construction applies with
$\GR=\GR(\cK)$ replaced by the (isomorphic) Grassmannian
$\GR^n=\GR(\cK^n)$, line bundles replaced by vector bundles and
$\cK=\gl_1(\cK)$ replaced by the loop algebra $\gl_n(\cK)$. The KP
flows on $\GR$ are now replaced by the multicomponent KP flows on
$\GR^n$, corresponding to maximal tori in $\gl_n(\cK)$ \cite{KvdL}
(see \cite{Plaza} and \cite{solitons} for more on the geometry of
multicomponent KP).

The Sato Grassmannian carries a canonical line bundle, the determinant
line bundle, which defines the Pl\"ucker embedding of $\GR$ into the
projective space of the fermionic Fock space.  This line bundle
restricts to the corresponding determinant or theta line bundles on
the moduli of curves and bundles. The actions of the Lie algebras
$\cK$, $\on{Der} \cK$ and $\gl_n(\cK)$ on the Grassmannian (and on the
corresponding moduli spaces) extend canonically to the actions of the
Heisenberg, Virasoro and Kac-Moody central extensions $\wh{\gl}_1$,
$Vir$ and $\wh{\gl}_n$ respectively on the determinant line bundles.

\subsection{$\Winfi$-orbits of Krichever data}\label{W-orbits}
The free fermion theory, the Sato Grassmannian and the determinant line bundle all 
carry an important additional symmetry: the action of
the Lie algebra $\cD(\cK)$ of differential operators with Laurent coefficients and
of its central extension $\Winfi$.

From the point of view of the KP hierarchy, the action of $\Winfi$ is given by the
Orlov-Schulman additional symmetries,
and is given explicitly on tau functions by the Adler-Shiota-van Moerbeke formula
(see \cite{vM} for a review).
Note that this symmetry algebra contains both
the Heisenberg and Virasoro algebras as zeroth-order and first-order operators. It
is natural, therefore, to ask 
for a moduli space interpretation of (integrated) orbits of the $\Winfi$ action
extending the 
Heisenberg-Virasoro uniformization of moduli of curves and bundles.

Thus, given a pointed curve $(X,x)$ we consider the moduli functor
$\BunDhat(X,x)$ of $\D$-line bundles on $X$ trivialized in the formal
neighborhood of $x$. By the Cannings-Holland correspondence (as in
Section \ref{D-bundles}) this functor is the direct limit of the moduli functors
of rank one torsion-free sheaves on cuspidal quotients of $X$ equipped
with trivializations near $x$ (in other words, we allow singularities
at all points other than $x$). It is then easy to see that the
Krichever embedding on cusp curves gives an embedding
$$\BunDhat(X,x)\hookrightarrow \GR.$$
Explicitly, this assigns to a
$\D$-line bundle $M$ trivialized near $x$ its de Rham cohomology on the
punctured curve $h(M|_{X\sm x})\subset h(M|_{D^\times})$, viewed as a subspace of its de Rham cohomology on the
punctured disc (which is itself identified with $\C((z))$).  The
determinant bundle on the Sato Grassmannian restricts to the usual
determinant line for torsion-free sheaves, which, as we have explained
in Section \ref{levels}, is the determinant line of de Rham cohomology
for $\D$-bundles.  Furthermore, we have:
\begin{prop} The $\Winfi$ flows on the Sato Grassmannian span the tangent space to the 
embedding $\BunDhat(X,x)\to \GR$ of the moduli of $\D$-line bundles on $X$, or
equivalently
rank one sheaves on cusp quotients of $X$, trivialized at $x$. 
\end{prop}
\begin{proof}
  Let $D$ and $D^{\times}$ denote the disc and punctured disc at $x$.
  Suppose $M$ is a $\D$-bundle equipped with a trivialization on $D$.
  First-order deformations of $M$ are given by $\Ext^1_\D(M,M)$; since
  $M$ is locally projective over $\D$, we find that $\Ext^1_\D(M,M) =
  H^1(X, \on{End}_\D(M))$.  Using the long exact sequence \bd
  0\rightarrow \underline{\End}_\D(M)\rightarrow
  \underline{\End}_\D(M)|_{X\smallsetminus x} \rightarrow
  \underline{\End}_\D(M)|_{X\smallsetminus
    x}/\underline{\End}_\D(M)\rightarrow 0 \ed and the isomorphisms
  \bd \underline{\End}_\D(M)|_{X\smallsetminus
    x}/\underline{\End}_\D(M) \cong \underline{\End}(M|_{D^\times}) /
  \underline{\End}(M|_D)\cong \D(\cK)/\D(\cO) \ed (the second one
  coming from the given trivialization of $M$ over $D$), we find that
  $$\on{H}^1(X,\underline{\End} (M))= \on{End}(M|_{X\sm x}) \bs
  \on{End}(M|_{D^\times}) / \on{End}(M|_D).$$
  Similarly, the tangent
  space to $\BunDhat(X,x)$ at a pair consisting of a $\D$-bundle $M$
  and trivialization on $D$ is $$\on{End}(M|_{X\sm x}) \bs \D(\cK).$$
  In particular the canonical action of $D(\cK)$ on $\BunDhat(X,x)$
  given by changing the transition functions on the punctured disc is
  infinitesimally transitive. Moreover, the de Rham functor identifies
  the action of $\D(\cK)$ on itself by right $\D$-module automorphisms
  with its action on $\cK$ as differential operators. Thus the
  embedding $\BunDhat(X,x)\to \GR$ is equivariant for $\D(\cK)$.
\end{proof}

\begin{remark}[Isomonodromy and moving the curve]
  We now have two interpretations of the action of vector fields
  $\Der(\cK)$ (or the Virasoro algebra) on Krichever data: on the one
  hand this action deforms the pointed curve $(X,x)$, but on the other
  hand it is tangent to the moduli $\BunDhat(X,x)$ of $\D$-line
  bundles on a {\em fixed} pointed curve $(X,x)$.  This apparent
  discrepancy is accounted for by the isomonodromy connection on the
  moduli spaces $\BunDhat(X,x)$ (see \cite{sug} for a parallel
  discussion of isomonodromy and the Segal-Sugawara construction).
  
  More precisely, as we vary $(X,x)$ infinitesimally, each $\D$-module
  on $X$ has a unique extension to a $\D$-module on the family, i.e.,
  a unique isomonodromic deformation. This gives a canonical
  horizontal distribution on the family of moduli of $\D$-bundles
  (independent of the trivialization near $x$) over the moduli
  $\fM_{g,1}$ of pointed curves. In coordinates (i.e., in terms of the
  Virasoro and $\Winfi$ uniformizations) this distribution is given by
  the embedding $\Der(\cK)\to \D(\cK)$.  Thus, the effect of moving
  the curve infinitesimally on the Krichever data in the Sato
  Grassmannian can be realized by torsion-free sheaves on cusp
  quotients of the fixed curve.
\end{remark}

\subsection{B\"acklund-Darboux transformations and the ad\`elic
Grassmannians}\label{backlund-darboux}
So far we have discussed only the {\em infinitesimal} symmetries of
the Sato Grassmannian and KP hierarchy.  We next turn our attention to
global symmetries, the B\"acklund-Darboux transformations or,
equivalently, Hecke modifications.

The global symmetry transformations of soliton hierarchies are the
B\"acklund-Darboux transformations, which in their classical form for
KdV involve replacing a Lax operator $L$ by $L'$ when the two are
conjugate by a differential operator $P$, i.e. $L=PQ$ and $L'=QP$ for
some differential operator $Q$. More generally, B\"acklund-Darboux
transformations (for equations of KdV type) are the residual
symmetries coming from the loop group symmetries of the corresponding
Grassmannians (see e.g. \cite{Palais}).

On the other hand, the global symmetry transformations of moduli
spaces of vector bundles or principal bundles on curves are the Hecke
correspondences. Two bundles on a curve are related by a (simple)
Hecke modification when we are given an isomorphism between their
restriction to the complement of a point (this is also the origin of
Tyurin coordinates on moduli of bundles).  Since the Hecke
correspondences are the residue of the loop group symmetries
underlying moduli of bundles, it is not surprising that
B\"acklund-Darboux transformations are expressed as Hecke
modifications on the level of geometric solutions of soliton
equations, though this connection does not appear to be very explicit
in the literature (see however \cite{LOZ}).

The Hecke modifications of line bundles are directly related to the
vertex operator of \cite{DJKM} (the bosonic realization of a free
fermion, see \cite{DJM,KNTY}), which is itself a ``Darboux operator in
disguise" \cite{vM}. As is explained in detail in \cite[Ch.20]{book},
the fermion vertex algebra, and specifically the fermion vertex
operator, is obtained as distributions on (i.e., as the group algebra
of) $\cK^\times/\cO^\times$ (whose $\C$-points are the integers
$\mathbf Z$). This group acts on the moduli of line bundles on a
pointed curve, with the generator taking a line bundle $\cL$ to its
Hecke transform $\cL(x)$. The corresponding element of the group
algebra (the delta function at the generator) gives the vertex
operator.

While Hecke modifications of line bundles form a group, Hecke
modifications for vector bundles or principal $G$-bundles do not (they
correspond to cosets, or double cosets, of a loop group). The
algebraic structure of composition of Hecke modifications is precisely
captured by the notion of factorization space, and specifically the
factorization space structure on the Beilinson-Drinfeld Grassmannian.
This is an ad\`elic version of the affine Grassmannian
$G(\cK)/G(\cO)$, which in the $GL_1$ case reduces to the geometry of
the group $\cK^\times/\cO^\times$ above. The infinitesimal version of
this structure is the Kac-Moody vertex algebra, which in the $GL_1$
case reduces to the Heisenberg algebra (i.e. the free boson, or the KP
flows), while the fermionic vertex operator itself generalizes to the
chiral Hecke algebra of Beilinson-Drinfeld (see \cite{book} for a
discussion).

Above we have described a factorization structure on the ad\`elic
Grassmannian of any rank, and shown that its infinitesimal version is
precisely the $\Winfi$ vertex algebra and its higher rank analogues
$\Winfi(\gl_n)$.  On the other hand, the global structure of the
ad\`elic Grassmannian captures the B\"acklund-Darboux transformations
of the KP hierarchy.  We have already seen that the ad\`elic
Grassmannian contains as a factorization subspace the
Beilinson-Drinfeld Grassmannian for $GL_n$ (see \eqref{BD in Grad}),
so in particular in the rank one case contains the information of the
fermionic vertex operator or simple B\"acklund-Darboux-Hecke
transformations.

The B\"acklund-Darboux transformations for the full KP hierarchy were
defined in \cite{BHY1}, where it is observed that the ad\`elic
Grassmannian (as originally defined by Wilson, as a space of
``conditions" \cite{Wilson bispectral}) is precisely the space of all
B\"acklund-Darboux transformations of the trivial solution. (See
\cite{HvdL} for a related study of B\"acklund-Darboux transformations
in terms of actions on the Grassmannian.) A point $V\subset\C((z))$ in
the Sato Grassmannian $\GR=\GR(\C((z)))$ is defined in \cite{BHY1} to
be a B\"acklund-Darboux transformation of a point $W\subset\C((z))$ if
there are polynomials $f,g$ in $z\inv$ such that
$$f\cdot V\subset W \subset g\inv \cdot V.$$
In the case when $W$
corresponds to an algebro-geometric solution $(X,x,\cL)$ of some n-KdV
hierarchy, i.e.  $z^{-n}$ extends to a global function on $X\sm x$
(for example the trivial solution $\C[z\inv]$), this implies that $V$
differs from $W$ by finite-dimensional conditions at some divisor $D$
on $X$. It then follows that $V$ is defined by a torsion-free sheaf on
a curve obtained by adding cusps to $X$ along $D$, and that this sheaf
is identified with $\cL$ off $D$. Equivalently, $V$ is defined by a
$\D$-bundle on $X$ which is identified with $\cL\ot \cD$ off
$D$---i.e. a Hecke modification of the $\D$-bundle $\cL\ot\cD$.
Conversely, all torsion-free sheaves on cusp quotients of $X$ (or
$\D$-bundles on $X$) define points $V$ which are B\"acklund-Darboux
transforms of the given solution $W$. In other words, the ad\`elic
Grassmannian of $\cL\ot\cD$ precisely parametrizes B\"acklund-Darboux
transforms of the corresponding KP solution.

Thus, we obtain a geometric approach to some of the results of
\cite{BHY2,BHY3} relating $\Winfi$, bispectrality and the ad\`elic
Grassmannian. (For a discussion of bispectrality for solutions of KP
and its multicomponent versions in terms of $\D$-bundles see
\cite{solitons}.)  In particular, the fact that the action of $\Winfi$
preserves the spaces of Darboux transformations of various solutions,
or that the corresponding tau functions form representations of
$\Winfi$, are explained by the fact that the $\Winfi$ vertex algebra
is the ``Lie algebra" of the ``group" of B\"acklund-Darboux
transformations, namely the ad\`elic Grassmannian of Hecke
modifications of the corresponding $\D$-bundle.

One advantage of the geometric approach is that it immediately extends
to the multicomponent KP hierarchies (see \cite{vdL} for a study of
the $\Winfi$ symmetries of these hierarchies).  Namely the
$\Winfi(\gl_n)$-action on algebro-geometric solutions exponentiates to
the space of B\"acklund-Darboux-Hecke transforms of the corresponding
$\D$-bundles, and in the case of solutions coming from the affine line
we obtain a bispectral involution. It would be interesting to find
explicit descriptions of the corresponding bispectral solutions and
the action of the $\Winfi$ symmetries on them.

\subsection{$\cW$-geometry}\label{W-geometry}
Conformal field theory has uncovered a fascinating new {\em
  $\cW$-geometry}.  This mathematically mysterious geometric structure
is expected to bear the same relation to the nonlinear $\cW_n$ vertex
algebras (the quantized symmetries of the $n$th KdV hierarchy) as the
moduli of bundles and curves bear to the Kac-Moody and Virasoro (or
$\cW_2$) algebras.  This geometry is further expected to have deep
connections to higher Teichm\"uller theory, isomonodromy, quantization
of higher Hitchin hamiltonians and geometric Langlands (for a sample
see \cite{Hull, Pope, LO}).

The present paper yields a precise mathematical framework for
$\cW$-geometry in the limiting, linear case of $\Winfi$ (studied for
example in \cite{Pope} and references therein), as we sketch below.
Namely, we propose that $\Winfi$ geometry is the geometry of a
noncommutative variety, the quantized cotangent bundles of a Riemann
surface $X$ (for general $X$).\footnote{In the following discussion we
  will suppress the distinctions between the different variants of the
  $\Winfi$ algebra, which differ by the central extension or the
  Heisenberg sub-algebra, whose geometric interpretations are evident.}
This is a natural consequence of the interpretation of the $\Winfi$
Lie algebra as the quantization of the Poisson algebra of functions on
the cotangent bundle of the punctured disc (or Hamiltonian vector
fields on the cylinder). Note that the quantized algebra generalizes
simultaneously the ring structure on functions and the Lie bracket on
Hamiltonian vector fields, hence the corresponding deformation problem
generalizes simultaneously the deformations of line bundles and of the
underlying variety.

Namely we consider the moduli stack $\BunD(X)$ of all $\D$-bundles on
$X$, equipped with the line bundle defined by the determinant of de
Rham cohomology, as a substitute for the moduli of bundles or curves
with $\Winfi$ symmetry. We have shown that the choice of
trivialization at a point $x$ defines a space $\BunDhat(X,x)$ (and
line bundle) which is embedded in the Sato Grassmannian (and
determinant bundle) and is uniformized by $\Winfi$. (An important
distinction from the moduli of bundles, parallel to the moduli of
torsion-free sheaves on a singular curve, is that not all $\D$-bundles
are locally trivial at a given $x$, though they are all generically
trivial.) The tangent space to $\BunD(X)$ at the trivial $\D$-bundle
is given by $\on{H}^1(X,\cD)$, and in fact functions on the formal
neighborhood of the trivial bundle are easily seen to be calculated as
the conformal blocks of the $\Winfi$ vertex algebra on $X$, as
expected from $\Winfi$ geometry.  Analogous statements hold at an
arbitrary $\D$-bundle $M$ with tangent space given by
$\on{H}^1(X,\on{End}(M))$ and formal functions calculated as conformal
blocks of the chiral algebra of endomorphisms of $M$. (More generally,
one can define $\D$-modules on $\BunD(X)$ from conformal blocks of any
representation of $\Winfi$.)

It would be very interesting to obtain a better geometric
understanding of the stack $\BunD(X)$ and its relation to the many
conjectured aspects of $\cW$-geometry.  In particular, for a curve of
genus larger than $1$ this stack is the noncommutative counterpart of
the stack of torsion-free sheaves on a Stein surface with no global
functions, and so may be expected to have good geometric properties.
(Note that the stacks that were studied in \cite{solitons, BGN}---the
Calogero-Moser spaces---are the moduli of {\em filtered} $\D$-bundles,
corresponding to a compactification of the noncommutative surface
under consideration.)  We also hope that this geometry might provide a
hint to the mystery of $\cW_n$-geometry.


\begin{thebibliography}{BGK2}

\bibitem[AMP]{AMP} A. \'{A}lvarez, J. Mu\~{n}oz and F. Plaza, The
  algebraic formalism of soliton equations over arbitrary base fields.
  In {\em Workshop on Abelian Varieties and Theta Functions
  (Morelia 1996)}, {\em Aportaciones Mat. Investig.} {\bf 13}, 3--40, Soc.
  Mat. Mexicana, 1998.  alg-geom/9606009.

\bibitem[AKDP]{ADKP}
E. Arbarello, C. De Concini, V. Kac,and C. Procesi, 
Moduli spaces of curves and representation theory.  {\em Comm. Math. Phys.}  {\bf 117}  (1988), 
no. 1, 1--36.


\bibitem[BHY1]{BHY1} B. Bakalov, E. Horozov and M. Yakimov,
B\"acklund--Darboux transformations in Sato's Grassmannian. arXiv:q-alg/9602010 
 {\em Serdica Math. J.}  {\bf 22}  (1996),  no. 4, 571--586.

\bibitem[BHY2]{BHY2} B. Bakalov, E. Horozov and M. Yakimov,
Bispectral algebras of commuting ordinary differential operators    
arXiv:q-alg/9602011.  {\em Comm. Math. Phys.}  {\bf 190}  (1997),  no. 2, 331--373.


\bibitem[BHY3]{BHY3} B. Bakalov, E. Horozov and M. Yakimov,  
Highest weight modules over $\Winfi$ algebra and the bispectral problem
  arXiv:q-alg/9602012.  {\em Duke Math. J.}  {\bf 93}  (1998),  no. 1, 41--72.  
  
\bibitem[BBE]{BBE} A. Beilinson, S. Bloch and H. Esnault,
  $\epsilon$-factors for Gauss-Manin determinants. arXiv:math/0111277.
  {\em Moscow Math. J.}  {\bf 2} (2002), no. 3, 477--532.

\bibitem[BD1]{Hecke} A. Beilinson and V. Drinfeld, Quantization of
Hitchin's integrable system and Hecke eigensheaves, available at {\tt
http://www.math.uchicago.edu/$\sim$mitya}.

\bibitem[BD2]{chiral} A. Beilinson and V. Drinfeld, {\em Chiral Algebras},
American Mathematical Society, Providence, 2004.


\bibitem[BS]{BS} A. Beilinson and V. Schechtman,
Determinant bundles and Virasoro algebras.
{\em Comm. Math. Phys.} {\bf 118} (1988), no. 4, 651--701. 

\bibitem[BF]{sug} D. Ben-Zvi and E. Frenkel, Geometric realization of the
Segal-Sugawara construction.
Topology, Geometry and Quantum Field Theory. Proc., 2002 Oxford
Symposium in Honour of the 60th Birthday of Graeme Segal. {\em LMS
Lecture Notes} {\bf 308} (2004) 46--97. arXiv:math.AG/0301206.




\bibitem[BN1]{cusps} D. Ben-Zvi and T. Nevins, Cusps and
$\D$--modules, {\em
J. Amer. Math. Soc.} {\bf 17} (2004), no. 1, 155--179.  arXiv:math.AG/0212094.

\bibitem[BN2]{solitons} D. Ben-Zvi and T. Nevins, $\D$-bundles and integrable
hierarchies,
arXiv:math/0603720.

\bibitem[BN3]{BGN} D. Ben-Zvi and T. Nevins, Perverse bundles and Calogero-Moser
spaces, {\em Compositio Math.}, to appear,
arXiv:math/0610097.


\bibitem[BW1]{BW automorphisms} Y. Berest and G. Wilson, Automorphisms
and ideals of the Weyl algebra, {\em Math. Ann.} {\bf
318} (2000), 127--147.  arXiv:math.QA/0102190.

\bibitem[BW2]{BW ideals} Y. Berest and G. Wilson, Ideal classes of the
  Weyl algebra and noncommutative projective geometry, With an
  appendix by Michel Van den Bergh, {\em Int. Math. Res. Not.} {\bf 2}
  (2002), 1347--1396.  arXiv:math.AG/0104248.

\bibitem[CH1]{CH ideals} R. Cannings and M. Holland, Right ideals of
rings of differential operators, {\em J. Algebra} {\bf 167} (1994),
116--141.


\bibitem[CH2]{CH cusps} R. Cannings and M. Holland, Limits of
compactified Jacobians and $\D$--modules on smooth projective curves,
{\em Adv. Math.} {\bf 135} (1998), 287--302.

\bibitem[DJKM]{DJKM} E. Date, M. Jimbo, M. Kashiwara and T. Miwa, 
Transformation groups for soliton equations.  Nonlinear integrable systems---
classical theory and quantum theory (Kyoto, 1981),  39--119, World Sci. Publishing,
Singapore, 1983.

\bibitem[DJM]{DJM} E. Date, M. Jimbo and T. Miwa,
Solitons. Differential equations, symmetries and infinite-dimensional algebras. 
Translated from the 1993 Japanese original by Miles Reid. Cambridge Tracts in 
Mathematics, 135. Cambridge University Press, Cambridge, 2000.

\bibitem[FB]{book} E. Frenkel and D. Ben-Zvi, {\em Vertex Algebras and
  Algebraic Curves}, Second edition. Mathematical Surveys and Monographs, 88. 
  American Mathematical Society, Providence, RI, 2004.
  
\bibitem[Ga]{Ga} D. Gaitsgory, Notes on 2D conformal field theory and string theory,
in {\em Quantum Field Theory for Mathematicians}, P. Deligne {\em et al.} eds., vol.
2, 1017--1089.

\bibitem[HvdL]{HvdL} G. Helminck and J. van de Leur, Geometric B\"acklund-Darboux 
transformations for the KP hierarchy. arXiv:solv-int/9806009 
{\em Publ. Res. Inst. Math. Sci.} {\bf 37}  (2001),  no. 4, 479--519. 

\bibitem[H]{Hull} C. Hull, ${\scr W}$-geometry.  
{\em Comm. Math. Phys.} {\bf 156} (1993), no. 2, 245--275. arXiv:hep-th/9211113

\bibitem[K]{Kac} V. Kac, Vertex algebras for beginners.  Second
  edition. University Lecture Series, 10. American Mathematical
  Society, Providence, RI, 1998.
 
\bibitem[KvdL]{KvdL} V. Kac and J. van de Leur, The $n$-component KP
  hierarchy and representation theory.  Integrability, topological
  solitons and beyond.  {\em J. Math. Phys.} {\bf 44} (2003), no. 8, 3245--3293.
  arXiv:hep-th/9308137


\bibitem[KV]{KV} M. Kapranov and E. Vasserot, Vertex algebras and the formal loop
space, {\em Publ. Math. IHES} {\bf 100} (2004), 
209--269.


\bibitem[KNTY]{KNTY}
N. Kawamoto, Y. Namikawa, A. Tsuchiya and Y. Yamada, 
Geometric realization of conformal field theory on Riemann surfaces.
{\em Comm. Math. Phys.} {\bf 116} (1988), no. 2, 247--308. 

\bibitem[Ko]{Ko} M. Kontsevich, 
The Virasoro algebra and Teichm\"uller spaces. (Russian)
{\em Funktsional. Anal. i Prilozhen.} {\bf 21} (1987), no. 2, 78--79. 

\bibitem[LO]{LO} A.Levin, M.Olshanetsky, 
Lie Algebroids and generalized projective structures on Riemann surfaces.
arXiv:0712.3828
  

\bibitem[LOZ]{LOZ} A. Levin, M. Olshanetsky and A. Zotov, Hitchin
  Systems - Symplectic Hecke Correspondence and Two-dimensional
  Version. arXiv:nlin/0110045.

\bibitem[Mu]{Mu} M. Mulase, Algebraic theory of the KP equations. In
{\em Perspectives in mathematical physics}, {\em
Conf. Proc. Lect. Notes Math. Phys.} {\bf III} (1994), 151--217.

\bibitem[Pa]{Palais} R. Palais, 
The symmetries of solitons.  {\em Bull. Amer. Math. Soc. (N.S.)}  {\bf 34}  (1997),  no. 4,
339--403.

\bibitem[Pl]{Plaza} F. Plaza Martin, Algebraic solutions of the
  multicomponent KP hierarchy, {\em J. Geom.
  Phys.} {\bf 36} (2000), no. 1--2, 1--21.  arXiv:math.AG/9907069.
  
\bibitem[Po]{Pope} C. Pope, Lectures on W algebras and W gravity. arXiv:hep-th/9112076
  
  
\bibitem[S]{Sa} M. Sato, The KP hierarchy and infinite-dimensional
  Grassmann manifolds. In {\em Theta functions---Bowdoin 1987, Part 1
  (Brunswick, ME, 1987)}, {\em Proc. Sympos. Pure Math.} {\bf
    49}, Part 1, 51--66, Amer. Math. Soc., Providence, RI, 1989.


\bibitem[SW]{SW} G. Segal and G. Wilson, Loop groups and equations of
  KdV type,  {\em Inst. Hautes \'Etudes Sci. Publ. Math.} {\bf 61}
  (1985), 5--65.

\bibitem[SS]{SS} S.~P. Smith and J.~T. Stafford, Differential
  operators on an affine curve, {\em Proc. London Math. Soc. (3)}
  {\bf 56} (1988), 229--259.

\bibitem[TUY]{TUY} A. Tsuchiya, K. Ueno and Y. Yamada, 
Conformal field theory on universal family of stable curves with gauge symmetries.  
Integrable systems in quantum field theory and statistical mechanics,  459--566, 
Adv. Stud. Pure Math., 19, Academic Press, Boston, MA, 1989.

\bibitem[U]{Ueno} K. Ueno, 
On conformal field theory.  Vector bundles in 
algebraic geometry (Durham, 1993),  283--345, London Math. Soc. Lecture Note Ser.,
208, 
Cambridge Univ. Press, Cambridge, 1995.

\bibitem[vdL]{vdL} J. van de Leur, The
  $\cW_{1+\infty}(\gl_S)$--symmetries of the $S$--component KP
  hierarchy.  hep-th/9411069.  {\em J. Math. Phys.}  {\bf 37} (1996), no. 5,
  2315--2337.

{\bf 104}  (1995),  no. 1, 32--42;  translation in {\em  Theoret. and Math.
Phys.}   {\bf 104}  (1995),  no. 1, 783--792 (1996).

\bibitem[vM]{vM} P. van Moerbeke, Alg\`ebres $\scr W$ et \'equations non-lin\'eaires. 
S\'eminaire Bourbaki. Vol. 1997/98.  Ast\'erisque  No. 252  (1998), Exp. No. 839, 3,
105--129.

\bibitem[W1]{Wilson bispectral} G. Wilson, Bispectral commutative
ordinary differential operators, {\em J. reine angew. Math.} {\bf 442}
(1993), 177--204.

\bibitem[W2]{Wilson CM} G. Wilson, Collisions of Calogero-Moser particles
and an adelic Grassmannian, {\em Invent. Math.} {\bf 133} (1998),
1--41.




\end{thebibliography}
\end{document}